\documentclass{amsart}
\usepackage[utf8]{inputenc}
\usepackage{amsthm}
\usepackage{amsmath}

\DeclareMathOperator\dist{dist}

\usepackage{enumerate}   
\usepackage{amssymb}
\usepackage{fancyhdr} 
\usepackage[left=3.5cm,right=3.5cm,top=3.5cm,bottom=3.5cm]{geometry}
\usepackage{xcolor}
\newcommand{\norm}[1]{\left\lVert#1\right\rVert}
\usepackage{graphicx}
\usepackage{ esint }
\newtheorem{theorem}{Theorem}[section]
\newtheorem{proposition}[theorem]{Proposition}
\newtheorem{corollary}[theorem]{Corollary}
\newtheorem{lemma}[theorem]{Lemma}

\theoremstyle{definition}
\newtheorem{example}[theorem]{Example}
\theoremstyle{remark}

\renewcommand\Re{\operatorname{Re}}
\renewcommand\Im{\operatorname{Im}}

\usepackage{url}
\usepackage[hidelinks]{hyperref}
\hypersetup{
	colorlinks=true,
	linkcolor=cyan,
	filecolor=cyan,
	citecolor =cyan,      
	urlcolor=cyan,
}
\author{Athanasios Kouroupis}
\address{Department of Mathematical Sciences, Norwegian University of Science and Technology (NTNU), 7491 Trondheim, Norway}
\email{athanasios.kouroupis@ntnu.no}
\title{Composition operators and generalized primes}

\begin{document}
\begin{abstract}
We study composition operators on the Hardy space $\mathcal{H}^2$ of Dirichlet series with square summable coefficients. Our main result is a necessary condition, in terms of a Nevanlinna-type counting function, for a certain class of composition operators to be compact on $\mathcal{H}^2$. To do that we extend our notions to a Hardy space  $\mathcal{H}_{\Lambda}^2$ of generalized Dirichlet series, induced in a natural way by a sequence of Beurling's primes.
\end{abstract}
\maketitle
\section{Introduction}
We consider the increasing sequence $\{p_n\}_{n\geq1}$ of primes and an arbitrary  increasing sequence $\{q_n\}_{n\geq1}$ satisfying the following:
\begin{enumerate}[(i)]
	\item The set $\{\log(p_n)\}_{n\geq1}\cup\{\log(q_n)\}_{n\geq1}$ is $\mathbb{Q}$-linear independent.
	\item $\{q_n\}_{n\geq1}$ is increasing, unbounded and each term is greater than $1$.
\end{enumerate}

For our purposes we will say that a real number $q>1$ is a generalized prime if the set $\{\log(p_n)\}_{n\geq1}\cup\{\log(q)\}$ is $\mathbb{Q}$-linear independent.

We will denote by  $\mathbb{N}_{p,q}=\{\lambda_n\}_{n\geq 1}$ the increasing sequence of numbers that can be written as a (unique) finite product of terms of the set $\{p_n\}_{n\geq0}\cup\{q_n\}_{n\geq1}$, i.e.

$$\lambda=p^aq^b:=p_1^{a_1}\cdot p_2^{a_2}\cdot\,\dots\,\cdot q_1^{b_1}\cdot q_2^{b_2}\cdot\dots\,.$$

A Dirichlet series is a function $g$ of the form
$$g(s)=\sum\limits_{n\geq 1}\frac{a_n}{n^s},\qquad s=\sigma+it.$$

The set of numbers $\mathbb{N}_{p,q}=\{\lambda_n\}_{n\geq 1}$ corresponds to generalized Dirichlet series, meaning function of the form
$$f(s)=\sum\limits_{n\geq 1}\frac{a_n}{\lambda_n^s},\qquad s=\sigma+it.$$
It is well-known that if a generalized Dirichlet series converges at a point $s_0=\sigma_0+it_0$, then it converges for every $s\in\mathbb{C}_{\sigma_0}$, where by $\mathbb{C}_\theta$ we denote the half-plane $\{z: \Re z\geq \theta\},\,\theta\in\mathbb{R}$.

The first to introduce such systems of numbers was Beurling \cite{B37}. Studying general Beurling's systems gives us a better understanding of the exceptional system of the classical primes. We refer the interested reader to \cite{D69,HL06,DMV06, ZHA07} for results related to number theory, like the prime number theorem and the Riemann hypothesis. 
Our point of view is more operator theoretical, a system of Beurling's primes naturally induces a Hardy space of generalized Dirichlet series, with frequencies  $\Lambda=\{\log \lambda_n\}_{n\geq 1}$ \cite{HEL69}. The idea behind using such systems is that the behavior of certain operators does not depend on the choice of primes.

The space $\mathcal{H}_{\Lambda}^2$ of generalized Dirichlet series with square summable coefficients is defined as 
$$\mathcal{H}_{\Lambda}^2=\left\{f(s)=\sum\limits_{n\geq 1}\frac{a_n}{\lambda_n^s}:\norm{f}^2_{\mathcal{H}_{\Lambda}^2}=\sum_{n\geq1}|a_n|^2<+\infty\right\}.$$

The Hardy space $\mathcal{H}^2$ \cite{HLS97} is the subspace of $\mathcal{H}_{\Lambda}^2$ containing all Dirichlet series,
$$\mathcal{H}^2=\left\{f(s)=\sum_{n\geq1}\frac{a_n}{n^s}:\norm{f}_{\mathcal{H}^2}^2=\sum_{n\geq1}|a_n|^2<+\infty \right\}.$$

Gordon and Hedenmalm \cite{GH99} determined the class $\mathfrak{G}$ of analytic functions $\psi:\mathbb{C}_{\frac{1}{2}}\rightarrow\mathbb{C}_{\frac{1}{2}}$ that induce  bounded  composition operators $C_\psi (f)=f\circ\psi$ on $\mathcal{H}^2$.
The class of symbols $\mathfrak{G}$ consists of all $\psi(s)=c_0s+\varphi(s)$, where $c_0$ is a non-negative integer and $\varphi$ is a Dirichlet series such that:
\begin{enumerate}[(i)]
	\item If $c_0=0$, then $\varphi(\mathbb{C}_0)\subset\mathbb{C}_\frac{1}{2}$.\label{itm1}
	\item If $c_0\geq 1$, then $\varphi(\mathbb{C}_0)\subset\mathbb{C}_0$ or $\varphi\equiv i\tau$ for some $\tau\in\mathbb{R}$.\label{itm2}
\end{enumerate}
Furthermore, a symbol $\psi\in\mathfrak{G}$ with $c_0\geq1$ induces a norm-one composition operator. We will use the notation $\mathfrak{G}_0$ and $\mathfrak{G}_{\geq 1}$ for the subclasses that satisfy \eqref{itm1} and \eqref{itm2}, respectively.

Defining the space $\mathcal{H}_{\Lambda}^2$, in some sense we added infinitely many prime-like numbers on the structure of $\mathcal{H}^2$. Our first aim is to prove that this does not have an effect on the behavior, meaning boundedness and compactness of a composition operator with symbol $\psi\in\mathfrak{G}_{\geq1}$.

\begin{theorem}\label{bddG1}
A symbol $\psi(s)=c_0s+\varphi(s)\in\mathfrak{G}_{\geq1}$, induces a bounded operator $C_\psi$ on $\mathcal{H}_{\Lambda}^2$ with norm $\norm{C_\psi}=1$.
\end{theorem}

\begin{theorem}\label{suf2}
	Suppose $\psi(s)=c_0s+\varphi(s)\in\mathfrak{G}_{\geq1}$ and that $C_\psi$ is a compact operator on $\mathcal{H}^2$. Then, $C_\psi$ is compact on $\mathcal{H}_{\Lambda}^2$.
\end{theorem}
In Section~\ref{sddp}, we work on the compactness of composition operators on the Hardy space $\mathcal{H}^2$. O. F. Brevig and K--M. Perfekt \cite{BP21} characterized compact composition operators on $\mathcal{H}^2$ with symbols in $\mathfrak{G}_0$. For symbols $\psi(s)=c_0s+\varphi(s)\in\mathfrak{G}_{\geq1}$, F. Bayart \cite{BAY21} gave the following sufficient condition for the composition operator $C_\psi$ to be compact
\begin{equation}\label{condition}
\lim\limits_{\Re w\rightarrow0^+}\frac{\mathcal{N}_{\psi}(w)}{\Re w}=0,
\end{equation}
where the Nevanlinna-type counting function $\mathcal{N}_{\psi}$ is defined as
$$\mathcal{N}_{\psi}(w)=\sum\limits_{\substack{s\in\psi^{-1}(\{w\})\\\Re s>0}}\Re s.$$

Conversely, Bailleul \cite{BL15} for finitely valent symbols, where $\phi$ is supported on a finite set of primes and Brevig and Perfekt \cite{BP20} under the assumption that $\phi$ is supported on single prime, proved that \eqref{condition} is also necessary for the composition operator $C_\psi$ to be compact. We say that a Dirichlet series $\phi$ is supported on a set of primes $\mathbb{P}$ if 
\begin{equation*}
\phi(s)=\sum\limits_{\substack{p|n\\ p\in\mathbb{P}}}\frac{a_n}{n^s}.
\end{equation*}

Our next result is a necessary condition without any additional assumption on the symbol $\psi\in\mathfrak{G}_{\geq1
}$. Specifically, we replace pointwise convergence in \eqref{condition} with  $L^1(\mathbb{T}^\infty)$ convergence. This answers a question posed by F. Bayart \cite[Question~3.6]{BAY21}.
\begin{theorem}\label{characterization}
	Suppose a symbol $\psi\in\mathfrak{G}_{\geq1}$ induces a  compact composition operator $C_\psi$ on $\mathcal{H}^2$. Then
	\begin{equation}\label{avcond}
	\lim\limits_{\Re w\rightarrow0}\int\limits_{\mathbb{T}^\infty}\frac{\mathcal{N}_{\psi_\chi}(w)}{\Re w}\,dm_\infty(\chi)=0.
	\end{equation}
\end{theorem}

The classical technique for proving such necessary conditions for compactness goes through the submean value property of the associated counting function and the behavior of the reproducing  kernels near the boundary, see for example \cite{SHAP87}. In Section~\ref{submeanv} we prove the weak submean value property for the average counting function $\int\limits_{\mathbb{T}^\infty}N_{\psi_\chi}(w)\,dm_\infty(\chi)$, Theorem~\ref{submean}. Using geometric function theory results, related to the distortion and the boundary behavior of conformal maps, we will be able to transfer our notions to the disk setting. The weak submean value property will then follow by classical results due to Shapiro \cite{SHAP87}.

The  main difficult in our setting is that reproducing kernels, $k_w(s)=\zeta(\overline{w}+s)$, on $\mathcal{H}^2$ are well defined only for points $w\in\mathbb{C}_{\frac{1}{2}}$.

F. Bayart have found an Example~\ref{exmp}, of a non-compact and bounded composition operator with symbol in $\mathfrak{G}_{\geq 1}$, that satisfies \eqref{avcond}. Theorem \ref{characterization} gives us the $L^1(\mathbb{T}^\infty)$ convergence of the quantity $\mathcal{N}_{\psi}(w)(\Re w)^{-1}\rightarrow 0$. It may be a step closer, but the characterization of compact composition operators, with symbols in $\mathfrak{G}_{\geq 1}$, remains an open problem \cite{BAY03, BAY21}.

For our purposes it was enough to study composition operators $C_\psi$ with symbols in the class $\mathfrak{G}_{\geq1}$. It would be interesting to have a characterization of the symbols that induce bounded composition operators on $\mathcal{H}^2_{\Lambda}$. 

\subsection*{Acknowledgments}
I would like to thank my supervisor, Karl--Mikael Perfekt, for his constant support and guidance. I am indebted to him and Ole Fredrik Brevig for the idea to add Beurling primes to the structure of the Hardy spaces of Dirichlet series.
Also, I would like to extend my gratitude to Frédéric Bayart for letting me include his Example~\ref{exmp} and for our fruitful mathematical conversations during my research visit at the Laboratoire de Mathématiques Blaise Pascal, Clermont-Ferrand. 

Part of the work has been conducted during a research visit at the Department of Mathematics in the Aristotle University of Thessaloniki.
\subsection*{Notation}
Throughout the article, we will be using the convention that $C$ denotes a positive constant which may vary from line to line. We will write that $C = C(\Omega)$ to indicate that the constant depends on a parameter $\Omega$. 
	\section{Background material}
\subsection{Composition operators in the disk setting}
The classical Hardy spaces $H^2$ consists of all holomorphic functions in the unit disk with square summable Taylor coefficients
$$H^2=\left\{f(s)=\sum_{n\geq0}a_nz^n:\norm{f}_{H^2}^2=\sum_{n\geq0}|a_n|^2<+\infty \right\}.$$

By the Littlewood subordination principle \cite{LIT25}, every holomorphic self-map of the unit disk, $\phi$, induces a bounded composition operator on $H^2$.
J. Shapiro in his seminal paper \cite{SHAP87} characterized the compact composition operator $C_\phi$ in terms of the Nevanlinna counting function
 $$N_{\phi}(z)=\sum\limits_{z_i\in \phi^{-1}(\{z\})}\log\frac{1}{|z_i|}, \qquad z \neq \phi(0).$$

The composition operator $C_\phi$ is compact on $H^2$ if and only if
\begin{equation}\label{eq:chunit}
\lim\limits_{|z|\rightarrow 1^-}\frac{N_\phi(z)}{\log\frac{1}{|z|} }=0.
\end{equation}

In order to prove the above theorem, J. Shapiro makes use of the Littlewood--Paley and the Stanton  formulae for the norm of a function $f\in H^2$ and its image $C_\phi(f)$, respectively.
\begin{equation}
\norm{f}_{H^2}^2=|f(0)|^2+\frac{2}{\pi}\int\limits_{\mathbb{D}}|f'(z)|^2\log\frac{1}{|z|}\,dA(z).
\end{equation}

\begin{equation}
\norm{C_\phi(f)}_{H^2}^2=|f\circ\phi(0)|^2+\frac{2}{\pi}\int\limits_{\mathbb{D}}|f'(z)|^2N_\phi(z)\,dA(z),
\end{equation}
where $dA(z)=dx\,dy,\,z=x+iy$, is the area measure.
\subsection{The infinite polytorus and vertical limits}
The infinite polytorus $\mathbb{T}^\infty$ is the countable infinite Cartesian product of copies of the unit  circle $\mathbb{T}$,
$$\mathbb{T}^\infty=\left\{\chi=(\chi_1,\chi_2,\dots): \,\chi_j\in\mathbb{T},\, j\geq 1\right\}.$$
As a compact abelian group with respect to coordinate-wise multiplication it posses a unique Haar measure $m_\infty$ \cite{R90}. We can identify the measure $m_\infty$ with the countable infinite product measure $m\times m\times\cdots$, where $m$ is the normalized Lebesgue measure of the unit circle.

The $\mathbb{Q}$-linear independence of the set $\{\log(p_n)\}_{n\geq 1}\cup \{\log(q_n)\}_{n\geq 1} $, implies that $\mathbb{T}^\infty$ is isomorphic to the group of characters of $((\mathbb{Q}_{p,q})_+,\cdot)$, where $(\mathbb{Q}_{p,q})_+$ are the fractions of $(\mathbb{N}_{p,q},\cdot)$. Given a point $\chi=(\chi_1,\chi_2,\dots)\in\mathbb{T}^\infty$, the corresponding character $\chi:(\mathbb{Q}_{p,q})_+\rightarrow\mathbb{T}$ is the completely multiplicative function on $\mathbb{N}_{p,q}$ such that $\chi(p_j)=\chi_{2j}$, $\chi(q_j)=\chi_{2j-1}$, extended to $(\mathbb{Q}_{p,q})_+$ through the relation $\chi(\lambda_n^{-1})=\overline{\chi(\lambda_n)}$. From now on we identify a point $\chi=(\chi_1,\dots)\in\mathbb{T}^\infty$ with the corresponding character $\chi(\lambda_n)$.

Suppose $f(s)=\sum\limits_{n\geq 1}\frac{a_n}{\lambda_n^s}$ and  $\chi(\lambda_n)$ is a character. The vertical limit function $f_\chi$ is defined as
$$f_\chi(s)=\sum\limits_{n\geq 1}\frac{a_n}{\lambda_n^s}\chi(\lambda_n).$$
Kronecker's theorem \cite{BOH34} justifies the name, since for any $\epsilon>0$, there exists a sequence of real numbers $\{t_j\}_{j\geq1}$ such that $f(s+it_j)\rightarrow f_\chi(s)$ uniformly on $\mathbb{C}_{\sigma_u(f)+\epsilon}$. The abscissae of convergence are defined likewise with the theory of Dirichlet series.
\begin{align*}
&\sigma_c(f)=\inf\left\{\sigma\in\mathbb{R}:\,f(s)=\sum\limits_{n\geq 1}\frac{a_n}{\lambda_n^\sigma}\qquad\text{converges}\right\},\\
&\sigma_a(f)=\inf\left\{\sigma\in\mathbb{R}:\,f(s)=\sum\limits_{n\geq 1}\frac{|a_n|}{\lambda_n^\sigma}\qquad\text{converges}\right\},\\
&\sigma_u(f)=\inf\left\{\sigma\in\mathbb{R}:\,f(s)=\sum\limits_{n\geq 1}\frac{a_n}{\lambda_n^s}\qquad\text{converges uniformly in}\qquad\mathbb{C}_\sigma \right\}.
\end{align*}
For a symbol $\psi(s)=c_0s+\varphi(s)\in\mathfrak{G}$ we set
$$\psi_\chi(s)=c_0s+\varphi_\chi(s),$$
and we observe that for every $\chi\in\mathbb{T}^\infty$ and $f\in\mathcal{H}_{\Lambda}^2$,
\begin{equation} \label{eq:comprule}
\left(C_\psi(f)\right)_\chi=f_{\chi^{c_0}}\circ\psi_\chi.
\end{equation}

Note that for a Dirichlet series $f\in\mathcal{H}^2\subset\mathcal{H}_{\Lambda}^2$ the vertical limit function has the form 
$$f_\chi(s)=\sum\limits_{n\geq1}\frac{a_n\chi(n)}{n^s},$$
where the character $\chi(n)$ exists in the dual group of $(\mathbb{Q}_+,\cdot)$, which is also isomorphic to $\mathbb{T}^\infty$.
\subsection{Hardy spaces on the infinite polytorus and the Bohr-lift}
We will make a short presentation of those topics and we refer to \cite{DGD19,HLS97,QQ20} for further information. For our purposes it would be enough to define only the spaces $H^2(\mathbb{T}^\infty)$ and $H^\infty(\mathbb{T}^\infty)$, but for expository reasons we will consider $1\leq p\leq \infty$. Let us first recall what happens in one dimension. The Hardy space $H^p,\,1\leq p\leq \infty$ consists of all functions in $L^p(\mathbb{T}, dm)$ with vanishing negative Fourier coefficients. 
The Fourier coefficient of a function $g\in L^1(\mathbb{T}^\infty)$ at a sequence $a=(a_1,a_2,\dots)\in \mathbb{Z}_0^\infty$  is defined as
$$\widehat{g}(a)=\int\limits_{\mathbb{T}^\infty}g(z)z^{-a}\,dm_\infty(z),$$
where $\mathbb{Z}_0^\infty$ is the set of all compactly supported sequences with integer terms and 
$$z^a=z_1^{a_1}\cdot z_2^{a_2}\cdot\dots,$$ is the multi-index notation. Similarly, we will denote by $\mathbb{N}_0^\infty$ the set of all compactly supported sequences of non-negative integers.

In a similar manner to the unit circle, the Hardy space $H^p(\mathbb{T}^\infty),\,1\leq p\leq \infty$ is defined as the subspace of  $L^p(\mathbb{T}^\infty)$, which contains all the functions with vanishing Fourier coefficients at sequences in $\mathbb{Z}_0^\infty\setminus\mathbb{N}_0^\infty$. 

By the definition of $\mathbb{N}_{p,q}=\{\lambda_n\}_{n\geq 1}$, for every $n\in\mathbb{N}$ there exist two unique sequences in $\mathbb{N}_0^\infty$, $\gamma_p(\lambda_n)$ and $\gamma_q(\lambda_n)$, such that 
$$\lambda_n=p^{\gamma_p(\lambda_n)}q^{\gamma_q(\lambda_n)}.$$

Starting with a generalized Dirichlet polynomial $f(s)=\sum\limits_{n\geq1}\frac{a_n}{\lambda_n^s}$ and mapping its prime term to a new variable, in the following way
$$p_i^{-s}\mapsto \chi_{2j},\qquad q_i^{-s}\mapsto \chi_{2j-1},\qquad i\in\mathbb{N},$$ 
we define the Bohr-lift of $f$ as
\begin{equation}\label{bohrlift}
B(f):=\sum\limits_{n\geq 1}a_n\chi(\lambda_n).
\end{equation}

The Bohr-lift is an isometric isomorphism between $\mathcal{H}_{\Lambda}^2$ and $H^2(\mathbb{T}^\infty)$. It is also, a norm preserving homeomorphism from $\mathcal{H}^\infty$ into $H^\infty(\mathbb{T}^\infty)$, see for example \cite{HLS97}. By $\mathcal{H}^\infty$ we denote  the space of all bounded Dirichlet series in $\mathbb{C}_0$, equipped with the uniform norm.

By Carleson theorem for $H^2(\mathbb{T}^\infty)$ \cite[Theorem~1.5]{HS03}, for every $f\in\mathcal{H}_{\Lambda}^2$ the series $B(f):=\sum\limits_{n\geq 1}a_n\chi(\lambda_n)$ converges for almost every character $\chi\in\mathbb{T}^\infty$. 

Thus, for every $f\in \mathcal{H}_{\Lambda}^2$  and for almost every $\chi\in \mathbb{T}^\infty$, we have that
$$\sigma_c(f_\chi)\leq 0.$$

\subsection{The ergodic theorem}
It is known \cite[Section~2.2]{QQ20} that given a sequence $\{a_n\}_{n\geq1}$ of $\mathbb{Q}$-linear independent real numbers, then the Kronecker flow $\{T_t\}_{t\in\mathbb{R}}$ is ergodic, where
\begin{equation}
T_t(\chi_1,\chi_2,\dots)=(e^{-ita_1}\chi_1,e^{-ita_2}\chi_2,\dots).
\end{equation}
By Birkhoff--Khinchin ergodic theorem, we obtain the following.
\begin{theorem}[\cite{CFS82,QQ20}]\label{ergodic}
If $g\in L^1(\mathbb{T}^\infty)$, then for almost every $\chi_0\in\mathbb{T}^\infty$,
	\begin{equation}\label{er1}
	\lim\limits_{T\rightarrow+\infty}\frac{1}{2T}\int\limits_{-T}^{T}g\left(T_t\chi_0\right)\,dt=\int\limits_{\mathbb{T}^\infty}g(\chi)\,dm_\infty(\chi).
	\end{equation}
	If $g$ is continuous, then \eqref{er1} holds for every  $\chi_0\in\mathbb{T}^\infty$.
\end{theorem}
Consequently, for every $f\in\mathcal{H}_{\Lambda}^2$ and for almost every character $\chi_0(\lambda_n)$
\begin{equation}\label{er2}
\lim\limits_{T\rightarrow+\infty}\frac{1}{2T}\int\limits_{-T}^{T}\left|f_{\chi_0}(it)\right|^2\,dt=\int\limits_{\mathbb{T}^\infty}\left|B(f)\right|^2\,dm_\infty\left(\chi\right).
\end{equation}
\subsection{The Littlewood--Paley and Stanton's formulae}
As in \cite[Lemma~2]{BAY03}, for every $f\in\mathcal{H}_{\Lambda}^2$ and $T>0$, we have the following Littlewood--Paley formula
\begin{equation}\label{litpal}
\norm{f}_{\mathcal{H}_{\Lambda}^2}^2=|f(+\infty)|^2+\frac{2}{T}\int\limits_{\mathbb{T}^\infty}\int\limits_{0}^{\infty}\int\limits_{-T}^{T}\left|f'_{\chi}(\sigma+it)\right|^2\sigma \,dt\,d\sigma \,dm_\infty(\chi).
\end{equation}
Suppose $\psi(s)=c_0s+\varphi(s) \in\mathfrak{G}_{\geq 1}$, by a non-injective change of variables \cite{SHAP87},
\begin{equation}\label{staton}
\norm{C_\psi(f)}^2=|f(+\infty)|^2+\frac{2}{\pi}\int\limits_{\mathbb{C}_0}\int\limits_{\mathbb{T}^\infty}\left|f_{\chi^{c_0}}'(w)\right|^2\mathcal{N}_{\psi_\chi}(w,T)  \,dm_\infty(\chi) \,dA(w),
\end{equation}
where the counting function $\mathcal{N}_{\psi_{\chi}}(w,T)$ is defined as
$$\mathcal{N}_{\psi_{\chi}}(w, T)=\frac{\pi}{T}\sum\limits_{\substack{s\in\psi_{\chi}^{-1}(\{w\})\\
		|\Im s|<T\\
		\Re s>0}}\Re s.$$
\section{When do composition operators change adding primes?}
In this section we will study the behavior of a composition operator $C_\psi,\,\psi\in\mathfrak{G}_{\geq 1}$ on the space $\mathcal{H}^2_{\Lambda}$. Our approach has been inspired by results in \cite[Section 3]{BP20}.

Let $\mathbb{N}_q=\{b_k\}_{k\geq 1}$ be the increasing sequence of numbers that can be written as a finite product of terms of the set $\{q_n\}_{n\geq1}$.
We observe that $b_i^{-s}\mathcal{H}^2\perp b_j^{-s}\mathcal{H}^2$, when $b_i\neq b_j$. Thus, $\mathcal{H}_{\Lambda}^2$ has the following orthogonal decomposition
\begin{equation}
\mathcal{H}_{\Lambda}^2=\bigoplus\limits_{k\geq1}b_k^{-s}\mathcal{H}^2.
\end{equation}
\begin{proposition}
	Let $\psi(s)=c_0s+\varphi(s)\in\mathfrak{G}_{\geq1}$ and $k\in\mathbb{N}$. Then, the composition operator $C_\psi$ maps $b_k^{-s}\mathcal{H}^2$ into $b_k^{-c_0s}\mathcal{H}^2$ and its restriction $C_{\psi, k}$ to $b_k^{-s}\mathcal{H}^2$ has norm $\norm{C_\psi}=1$.
\end{proposition}
\begin{proof}
First, we observe that
	\begin{equation*}
	C_\psi(b_k^{-s}n^{-s})=b_k^{-sc_0}b_k^{-\varphi(s)}C_{\psi}(n^{-s})\in b_k^{-sc_0}\mathcal{H}^2.
	\end{equation*}
The Bohr--lift respects multiplication \cite{HLS97}, that is
\begin{equation*}
B\left(mf\right)=B\left(m\right)B\left(f\right),\qquad m\in\mathcal{H}^\infty,\qquad f\in\mathcal{H}^2.
\end{equation*}
For every Dirichlet polynomial $f$, we have that
\begin{equation*}
\norm{C_\psi(b_k^{-s}f)}_{\mathcal{H}_{\Lambda}^2}^2=\norm{b_k^{-sc_0}b_k^{-\varphi(s)}C_{\psi}(f)}_{\mathcal{H}_{\Lambda}^2}^2=\int\limits_{\mathbb{T}^\infty}\left|B\left(b_k^{-\varphi(s)}\right)B\left(C_{\psi}(f)\right)\right|^2\,dm_\infty(\chi).
\end{equation*}
As we have already discussed, the Bohr--lift is norm preserving between $\mathcal{H}^\infty$ and $H^\infty(\mathbb{T}^\infty)$. Therefore
\begin{equation*}
\norm{B\left(b_k^{-\varphi(s)}\right)}_{H^\infty(\mathbb{T}^\infty)}=\norm{b_k^{-\varphi(s)}}_{\mathcal{H}^\infty}\leq 1.
\end{equation*}
Thus 
\begin{equation*}
	\norm{C_\psi(b_k^{-s}f)}_{\mathcal{H}_{\Lambda}^2}\leq \norm{C_\psi(f)}_{\mathcal{H}^2}\leq \norm{b_k^{-s}f}_{\mathcal{H}_{\Lambda}^2}.\qedhere
\end{equation*}	

\end{proof}
\begin{corollary}\label{corbdd}
	Let $\psi(s)=c_0s+\varphi(s)\in\mathfrak{G}_{\geq1}$. Then, the induced composition operator on $\mathcal{H}_{\Lambda}^2$ has the following orthogonal decomposition
	\begin{equation}\label{ordec}
	C_\psi=\bigoplus\limits_{k\geq0}C_{\psi, k}.
	\end{equation}
\end{corollary}
\begin{proof}[\textbf{Proof of Theorem \ref{bddG1}}]
The proof follows directly from the Corollary~\ref{corbdd}.
\end{proof}
\begin{proof}[\textbf{Proof of Theorem \ref{suf2}}]
	By \eqref{ordec}, it is sufficient to prove the following:
	\begin{enumerate}[(i)]
		\item $\norm{C_{\psi,k}}\rightarrow0.$\label{item1}
		\item $C_{\psi,k}$ is compact for every $k\geq0$.\label{item2}
	\end{enumerate}
	First we will prove \eqref{item1}. By Theorem \ref{ergodic}, for every Dirichlet polynomial $f\in\mathcal{H}^2$ and for almost every $\chi_0\in\mathbb{T}^\infty$, we have that
	\begin{multline*}
	\norm{C_{\psi,k}(b_k^{-s}f)}_{\mathcal{H}_{\Lambda}^2}^2=\int\limits_{\mathbb{T}^\infty}\left|B\left(b_k^{-\psi}C_\psi(f)\right)(\chi)\right|^2\,dm_\infty(\chi)\\
	=\lim\limits_{T\rightarrow+\infty}\frac{1}{2T}\int\limits_{-T}^{T}b_k^{-2\Re\left(\psi_{\chi_0}(it)\right)}\left|B\left(C_\psi(f)\right)(\lambda_n^{-it}\chi_0)\right|^2\,dt.
	\end{multline*}

The symbol $\psi$ has boundary values $\psi_{\chi}(it)=\lim\limits_{\sigma\rightarrow0^+}\psi_\chi(\sigma+it)$ for almost every $t\in\mathbb{R}$ and for almost every $\chi\in\mathbb{T}^\infty$. Furthermore, the vertical limit $\psi_\chi$ is in the class $\mathfrak{G}_{\geq1}$, see \cite{BAY02,BP20}. Thus
	\begin{align}\label{comweak}
	\norm{C_{\psi,k}(b_k^{-s}f)}_{\mathcal{H}_{\Lambda}^2}^2&\leq \lim\limits_{T\rightarrow+\infty}\frac{1}{2T}\int\limits_{-T}^{T}\lfloor b_k\rfloor^{-2\Re\left(\psi_{\chi_0}(it)\right)}\left|B\left(C_\psi(f)\right)(\lambda_n^{-it}\chi_0)\right|^2\,dt\nonumber\\
	&=\norm{C_{\psi}(\lfloor b_k\rfloor^{-s}f)}_{\mathcal{H}^2}^2,
	\end{align}
	where $\lfloor\cdot\rfloor$ is the floor function. We assume that \eqref{item1} fails, without loss of generality there exist $\delta>0$ and a sequence of Dirichlet polynomials $\{f_k\}_{k\geq1}$ in the unit ball of $\mathcal{H}^2$ such that
	\begin{equation}\label{contr}
	\norm{C_{\psi}(\lfloor b_k\rfloor^{-s}f_k)}_{\mathcal{H}^2}\geq\norm{C_\psi(b_k^{-s}f_k)}_{\mathcal{H}_{\Lambda}^2}>\delta,\qquad k\in\mathbb{N}.
	\end{equation}
	The sequence $\{\lfloor b_k\rfloor^{-s}f_k\}_{k\geq1}$ converges weakly to $0$ in $\mathcal{H}^2$ and as consequence
	$$\lim\limits_{n\rightarrow+\infty}\norm{C_{\psi}(\lfloor b_k\rfloor^{-s}f_k)}_{\mathcal{H}^2}=0.$$
	This contradicts with \eqref{contr}. Therefore, 
	$$\norm{C_{\psi,k}}\rightarrow0.$$
	For \eqref{item2}, we consider an arbitrary sequence $\{b_k^{-s}g_j\}_{j\geq1}$, which converges weakly to $0$ and we observe that $\{g_j\}_{j\geq1}$ is also weakly convergent to $0$ in $\mathcal{H}^2$. This implies that
	$$\norm{C_\psi(b_k^{-s}g_j)}_{\mathcal{H}_{\Lambda}^2}\leq\norm{C_\psi(g_j)}_{\mathcal{H}^2}\rightarrow0.$$
	Thus, $C_{\psi,k}$ is compact for every $k\geq1$.
\end{proof}
\section{Symbols that do not depend on a prime}\label{sddp}
\subsection{Submean value property}\label{submeanv}
Let $\Omega$ be an open subset of $\mathbb{C}$. We say that a function $u:\Omega\rightarrow[-\infty,\infty)$ satisfies the submean value property if for every disk $\overline{D(w,r)}\subset\Omega$
\begin{equation*}
u(w)\leq \frac{1}{|D(w,r)|}\int\limits_{D(w,r)}u(z)\,dA(z),
\end{equation*}
where $|D(w,r)| = \pi r^2$ is the area of the disk.

Shapiro \cite[Section 4]{SHAP87} proved that for every holomorphic self-map of the unit disk $\phi$, the Nevanlinna counting function $N_\phi$ satisfies the submean value property in $\mathbb{D}\setminus\{\phi(0)\}$.

The aim of this subsection is to prove the weak submean value property Theorem \ref{submean} for the average $\int\limits_{\mathbb{T}^\infty}	\mathcal{N}_{\psi_\chi}(w)\,dm_\infty(\chi)$, where $\psi\in\mathfrak{G}_{\geq 1}$.
\begin{equation}
\int\limits_{\mathbb{T}^\infty}	\mathcal{N}_{\psi_\chi}(w)\,dm_\infty(\chi)\leq \frac{C}{|D(w,r)|}\int\limits_{D(w,r)}\int\limits_{\mathbb{T}^\infty}	\mathcal{N}_{\psi_\chi}(z)\,dm_\infty(\chi)\,dA(z).
\end{equation}

Our argument will rely on a technique which has been developed in \cite{BP21,KP22} and allows us to transfer our notions in the disk setting.

We consider the unique conformal map $F$ from the unit disk onto the rectangle $$R=\{z:|\Im z|<2,\, 0<\Re z<2\},$$ with $F(0)=1$ and $F'(0)>0$. 
\begin{lemma}\label{estkoebe}
Suppose $s$ is a point with $0<\Re s< 1$ and $|\Im s|<2$. Then
\begin{equation}\label{estkoebe1}
1-|F^{-1}(s)|^2\leq C\Re s.
\end{equation}
Furthermore, if $0<\Re s< 1$ and $|\Im s|<1$. Then
\begin{equation}\label{estkoebe2}
1-|F^{-1}(s)|^2\geq C\Re s.
\end{equation}
\end{lemma}
\begin{proof}
By the Koebe quarter theorem \cite[Corollary~1.4]{POM92}, for every $s\in R$, we have that
\begin{equation}\label{quarter}
\frac{1-|F^{-1}(s)|^2}{4\left|\left(F^{-1}\right)'(s)\right|}\leq \dist(s,\partial R)\leq \frac{1-|F^{-1}(s)|^2}{\left|\left(F^{-1}\right)'(s)\right|}.
\end{equation}
By the Caratheodory \cite[Theorem~2.6]{POM92} and the Kellogg-Warschawski theorems \cite[Theorem~ 3.9]{POM92}, there exist absolute constants $\delta_1,\,\delta_2>0$ such that for $0<\Re s< 1$ and $|\Im s|<1$ 
$$0<\delta_1<|\left(F^{-1}\right)' |<\delta_2<\infty.$$
This and \eqref{quarter} imply \eqref{estkoebe2}.

Again, by the Kellogg-Warschawski theorem there exists $r>0$ such that $|\left(F^{-1}\right)'(s) |$ is bounded in $\overline{R}\cap D\left(F^{-1}(\pm 2i),r\right)$. Now, 
\eqref{estkoebe1} follows by the Koebe quarter theorem working as above.
\end{proof}
\begin{lemma}[\cite{KP22}]\label{bddsubmean}
Let $\Omega$ be a bounded subdomain of $\mathbb{C}$ and $\phi:\mathbb{D}\rightarrow\Omega$ be holomorphic. Then, the classical Nevanlinna counting function $N_{\phi}(w)$ satisfies the submean value property.
\end{lemma}

\begin{lemma}\label{submeanforbdd}
	Let $\psi\in\mathfrak{G}_{\geq 1}$. Then, there exists an absolute constant $C>0$ such that
	\begin{equation}
	\mathcal{N}_{\psi}(w,1)\leq \frac{C}{|D(w,r)|}\int\limits_{D(w,r)}\mathcal{N}_{\psi}(z,2)\,dA(z),
	\end{equation}
	for every disk $D(w,r)\subset\mathbb{C}_0\setminus\mathbb{C}_{\frac{1}{2}}$.
\end{lemma}
\begin{proof}
Let $F_{\sigma}(z)= F(z)+\sigma$ be the Riemann map from the unit disk onto the rectangle
	$$R_{\sigma}=\left\{z: \sigma<\Re z<2+\sigma,\, |\Im z|<2 \right\},$$
	with $F_{\sigma}(0)=1+\sigma$ and $F_{\sigma}'(0)>0$.
	
By Lemma \ref{estkoebe}
\begin{equation}\label{skoebe1}
1-|F_\sigma^{-1}(s)|^2\leq C_1(\Re s-\sigma),
\end{equation}
whenever $\sigma<\Re s< 1$, $|\Im s|<2$ and
\begin{equation}\label{skoebe2}
1-|F_\sigma^{-1}(s)|^2\geq C_2(\Re s-\sigma),
\end{equation}
whenever $\sigma<\Re s< 1$, $|\Im s|<1$.

We observe that $\Re s\leq \Re \psi( s)$ and that  $$1-\left|F_{\sigma}^{-1}(s)\right|^2\sim \log\frac{1}{\left|F_{\sigma}^{-1}(s)\right|},\qquad s\in R_\sigma \cap\mathbb{C}_0\setminus\mathbb{C}_{\frac{1}{2}}.$$
By \eqref{skoebe2}, for $z\in D(w,r)\subset\mathbb{C}_0\setminus\mathbb{C}_{\frac{1}{2}}$
\begin{multline*}
\mathcal{N}_{\psi}(z,1,2\sigma):=\pi\sum\limits_{\substack{s\in\psi^{-1}(\{z\})\\
			|\Im s|<1\\
		\Re s>2\sigma}}\Re s=\pi\sum\limits_{\substack{s\in\psi^{-1}(\{z\})\\
		|\Im s|<1\\
		2\sigma<\Re s<\frac{1}{2}}}\Re s\\
\leq 2\pi \sum\limits_{\substack{s\in\psi^{-1}(\{z\})\\
			|\Im s|<1\\
		\sigma<\Re s<\frac{1}{2}}} \left(\Re s-\sigma\right)
\leq C \sum\limits_{\substack{s\in\psi^{-1}(\{z\})\\
			|\Im s|<2\\
		\sigma<\Re s<\frac{1}{2}}}\left(1-\left|F_{\sigma}^{-1}(s)\right|^2\right)
	\leq C N_{\psi\circ F_{\sigma}}(z).
\end{multline*}
By \eqref{skoebe1}, for  $z\in D(w,r)$, we have that
	\begin{align*}
	N_{\psi\circ F_{\sigma}}(z)\leq C \sum\limits_{\substack{s\in\psi^{-1}(\{z\})\\
			|\Im s|<2\\
			\sigma<\Re s<\frac{1}{2}}}\left(1-\left|F_{\sigma}^{-1}(s)\right|^2\right) 
	\leq C\frac{\pi}{2 }\sum\limits_{\substack{s\in\psi^{-1}(\{z\})\\
			|\Im s|<2\\
			\sigma<\Re s<\frac{1}{2}}}\Re s= C \mathcal{N}_{\psi}(w,2,\sigma).
	\end{align*}
By Lemma \ref{bddsubmean} the function $N_{\psi\circ F_{\sigma}}$ satisfies the submean value property and 
	\begin{equation*}
	\mathcal{N}_{\psi}(z,1,2\sigma)\leq C_1N_{\psi\circ F_{\sigma}}(z)\leq C_2\mathcal{N}_{\psi}(z,2,\sigma).
	\end{equation*}
	Therefore 
	\begin{equation}\label{almostsubmean}
	\mathcal{N}_{\psi}(w,1,2\sigma)\leq  \frac{C}{|D(w,r)|}\int\limits_{D(w,r)}\mathcal{N}_{\psi}(z,2,\sigma)\,dA(z).
	\end{equation}
We can apply the monotone convergence theorem  to let $\sigma\rightarrow0^+$, yielding that
$$\mathcal{N}_{\psi}(w,1)\leq \frac{C}{|D(w,r)|}\int\limits_{D(w,r)}\mathcal{N}_{\psi}(z,2)\,dA(z),$$
	for an absolute constant $C>0$.
\end{proof}
The following theorem will allow us to apply Theorem \ref{ergodic} for the counting function $\mathcal{N}_{\psi_\chi}(w)$. 
\begin{theorem}{\cite{BAY03}}\label{bayartL1}
	Let $\psi(s)=c_0s+\varphi(s)\in \mathfrak{G}_{\geq1}$. Then, for every $w\in\mathbb{C}_0$
	\begin{equation}
	\mathcal{N}_\psi(w)\leq\frac{\Re w}{c_0}.
	\end{equation}
\end{theorem}
The following lemma will be of key importance for the proof of the weak submean value property, Theorem~\ref{submean}, and Theorem~\ref{suf1}. Despite its technical and maybe serendipitous look, the idea behind Lemma~\ref{tonellitrick} may be useful. See, for example the interchange of limits problem \cite[Problem~1]{BP21} and the partial solution of it \cite[Theorem~4.9]{KP22}.
\begin{lemma}\label{tonellitrick}
Let $\psi\in\mathfrak{G}_{\geq1},\,T>0$ and $w\in\mathbb{C}_0$. Then
\begin{equation}
\int\limits_{\mathbb{T}^\infty}	\mathcal{N}_{\psi_\chi}(w)\,dm_\infty(\chi)=\int\limits_{\mathbb{T}^\infty}	\mathcal{N}_{\psi_\chi}(\Re w)\,dm_\infty(\chi)=\frac{1}{2\pi c_0}\int\limits_{\mathbb{T}^\infty}	\int\limits_{-\infty}^{+\infty}\mathcal{N}_{\psi_\chi}(w+it,T)\,dt\,dm_\infty(\chi).
\end{equation}
\end{lemma}
\begin{proof}
We observe that $s\in\psi_{\chi}^{-1}(\{w+it\})$ if and only if $s-\frac{it}{c_0}\in\psi_{\chi\chi_t}^{-1}(\{w\}),$ where $\chi_t(n)=n^{-\frac{it}{c_0}}$ and $t\in\mathbb{R}$. Therefore
\begin{align*}
\int\limits_{\mathbb{T}^\infty}	\int\limits_{-\infty}^{+\infty}\mathcal{N}_{\psi_\chi}(w+it,T)\,dt\,dm_\infty(\chi)&=\frac{\pi}{T}\int\limits_{\mathbb{T}^\infty}	\int\limits_{-\infty}^{+\infty}\sum\limits_{\substack{s\in\psi_{\chi\chi_t}^{-1}(\{w\})\\
		-T-\frac{t}{c_0}<\Im s<T-\frac{t}{c_0}\\
		\Re s>0}}\Re s \,dt\,dm_\infty(\chi).
\end{align*}
The Haar measure $m_\infty$ is rotation invariant. This and Tonelli's theorem imply that 
\begin{align*}
\int\limits_{\mathbb{T}^\infty}	\int\limits_{-\infty}^{+\infty}\mathcal{N}_{\psi_\chi}(w+it,T)\,dt\,dm_\infty(\chi)&=\frac{\pi}{T}\int\limits_{\mathbb{T}^\infty}	\int\limits_{-\infty}^{+\infty}\sum\limits_{\substack{s\in\psi_{\chi}^{-1}(\{w\})\\
		-T-\frac{t}{c_0}<\Im s<T-\frac{t}{c_0}\\
		\Re s>0}}\Re s \,dt\,dm_\infty(\chi)\\
&=\frac{\pi}{T}\int\limits_{\mathbb{T}^\infty}	\sum\limits_{\substack{s\in\psi_{\chi}^{-1}(\{w\})\\
			\Re s>0}}\Re s\int\limits_{c_0(-\Im s-T)}^{c_0(T-\Im s)} \,dt\,dm_\infty(\chi)\\
&=2c_0 \pi \int\limits_{\mathbb{T}^\infty}	\mathcal{N}_{\psi_\chi}(w)\,dm_\infty(\chi)
=2c_0 \pi \int\limits_{\mathbb{T}^\infty}	\mathcal{N}_{\psi_\chi}(\Re w)\,dm_\infty(\chi). \qedhere
\end{align*}
\end{proof}
\begin{theorem}\label{submean}
	Let $\psi\in\mathfrak{G}_{\geq 1}$. Then, there exists an absolute constant $C>0$ such that
	\begin{equation}
	\int\limits_{\mathbb{T}^\infty}	\mathcal{N}_{\psi_\chi}(w)\,dm_\infty(\chi)\leq  \frac{C}{|D(w,r)|}\int\limits_{D(w,r)}\int\limits_{\mathbb{T}^\infty}	\mathcal{N}_{\psi_\chi}(z)\,dm_\infty(\chi)\,dA(z),
	\end{equation}
	for every disk $D(w,r)\subset\mathbb{C}_0\setminus\mathbb{C}_{\frac{1}{2}}$.
\end{theorem}
\begin{proof}
By Lemma \ref{submeanforbdd} 
	\begin{equation}\label{18}
\mathcal{N}_{\psi_\chi}(w+it,1)\leq \frac{C}{|D(w,r)|}\int\limits_{D(w,r)}\mathcal{N}_{\psi_\chi}(z+it,2)\,dA(z).
\end{equation}
The proof follows by Lemma \ref{tonellitrick} integrating \eqref{18} with respect to $\chi\in\mathbb{T}^\infty$ and then $t\in\mathbb{R}$.
\end{proof}
\subsection{Necessity when omitting a prime}
This subsection is devoted to the following weaker version of Theorem~\ref{characterization}.
\begin{theorem}\label{suf1}
	Suppose $\psi(s)=c_0s+\varphi(s)\in\mathfrak{G}_{\geq 1}$ with $\varphi(s)=\sum\limits_{p\nmid n}\frac{a_n}{n^s}$, where $p$ is a prime number. If the induced composition operator is compact on $\mathcal{H}^2$, then
	\begin{equation*}
	\lim\limits_{\Re w\rightarrow0}\int\limits_{\mathbb{T}^\infty}\frac{\mathcal{N}_{\psi_\chi}(w)}{\Re w}\,dm_\infty(\chi)=0.
	\end{equation*}
\end{theorem}
To prove the theorem we will use a variant of the classical technique, which gives necessary conditions for compactness. But, it is worth mentioning the ideas behind the steps of the proof. First, we considered a symbol that does not depend on a prime. We did that in order to separate the derivative of the reproducing kernel and the counting function  under the integral sign on the infinite polytorus, \eqref{eq:sep}. Then, using Lemma~\ref{tonellitrick}, the average counting function arises, \eqref{eq:usetrick}. In the last step of the proof, we make use of the translation invariance of that average, to start with an integral on the real line and then introduce a proper disk to gain an additional power of $\Re s_n$ and derive the necessary inequality, \eqref{eq:final}.
\begin{proof}
Without loss of generality we assume that $p=2$. Let $\{s_n\}_{n\geq1}\subset \mathbb{C}_0$ be an arbitrary sequence such that $\Re s_n\rightarrow0^+$. We observe that the induced sequence $\{K_{s_n,2}\}_{n\geq1
}$ of normalized reproducing kernels associated to the prime $2$, defined as
$$K_{s_n,2}(s)=\sqrt{1-4^{-\Re s_n}}\sum\limits_{n\geq0}\frac{1}{2^{n(\overline{s_n}+s)}}$$ converges weakly to $0$, as $n\rightarrow\infty$. Therefore 
\begin{equation}\label{comp0}
\lim\limits_{n\rightarrow+\infty}\norm{C_\psi(K_{s_n,2})}=0.
\end{equation}
Stanton's formula \eqref{staton} yields to the following
\begin{align}\label{eq:sep}
\norm{C_\psi(K_{s_n,2})}^2&\geq C\int\limits_{\mathbb{C}_0}\int\limits_{\mathbb{T}^\infty}\left|(K_{s_n,2})_{\chi^{c_0}}'(w)\right|^2\mathcal{N}_{\psi_\chi}(w,1)\,dm_\infty(\chi)\,dA(w)\nonumber\\
&\geq C\int\limits_{\mathbb{C}_0}\int\limits_{\mathbb{T}}\left|(K_{s_n,2})_{\chi_1^{c_0}}'(w)\right|^2\int\limits_{\mathbb{T}^\infty}\mathcal{N}_{\psi_\chi}(w,1)\,dm_\infty(\chi)\,dA(w).
\end{align}
By Parseval's formula and Lemma \ref{tonellitrick}
\begin{align*}
\norm{C_\psi(K_{s_n,2})}^2&\geq C(1-4^{-\Re s_n})\int\limits_{\mathbb{C}_0}\sum\limits_{n\geq1}n^24^{-n(\Re s_n+\Re s)}\int\limits_{\mathbb{T}^\infty}\mathcal{N}_{\psi_\chi}(w,1)\,dm_\infty(\chi)\,dA(w)\\
&\geq C(1-4^{-\Re s_n}) \int\limits_{0}^{+\infty}\sum\limits_{n\geq1}n^24^{-n(\Re s_n+\sigma)}\int\limits_{\mathbb{T}^\infty}\mathcal{N}_{\psi_\chi}(\sigma)\,dm_\infty(\chi)\,d\sigma\\
&\geq C \int\limits_{\frac{\Re s_n}{2}}^{\frac{3\Re s_n}{2}}\frac{1-4^{-\Re s_n}}{\left(1-4^{-(\Re s_n+\sigma)}\right)^3}\int\limits_{\mathbb{T}^\infty}\mathcal{N}_{\psi_\chi}(\sigma)\,dm_\infty(\chi)\,d\sigma.
\end{align*}
For sufficiently large $n\in\mathbb{N}$, we have that for every $t\in\mathbb{R}$
\begin{multline}\label{eq:usetrick}
\norm{C_\psi(K_{s_n,2})}^2\geq C \left(\Re s_n\right)^{-2}\int\limits_{\frac{\Re s_n}{2}}^{\frac{3\Re s_n}{2}}\int\limits_{\mathbb{T}^\infty}\mathcal{N}_{\psi_\chi}(\sigma)\,dm_\infty(\chi)\,d\sigma\\=C \left(\Re s_n\right)^{-2}\int\limits_{\frac{\Re s_n}{2}}^{\frac{3\Re s_n}{2}}\int\limits_{\mathbb{T}^\infty}\mathcal{N}_{\psi_\chi}(\sigma+it)\,dm_\infty(\chi)\,d\sigma.
\end{multline}
Thus
\begin{align}\label{eq:gainpower}
\norm{C_\psi(K_{s_n,2})}^2&\geq C \left(\Re s_n\right)^{-3}\int\limits_{\frac{\Re s_n}{2}}^{\frac{3\Re s_n}{2}}\int\limits_{-\sqrt{(\frac{\Re s_n}{2})^2-(\sigma-\Re s_n)^2}}^{\sqrt{(\frac{\Re s_n}{2})^2-(\sigma-\Re s_n)^2}}\int\limits_{\mathbb{T}^\infty}\mathcal{N}_{\psi_\chi}(\sigma+it)\,dm_\infty(\chi)\,dt \,d\sigma\nonumber\\
&\geq C (\Re s_n)^{-1}\frac{1}{\left|D(\Re s_n, \frac{\Re s_n}{2})\right|}\int\limits_{D(\Re s_n, \frac{\Re s_n}{2})}\int\limits_{\mathbb{T}^\infty}\mathcal{N}_{\psi_\chi}(w)\,dm_\infty(\chi)\,dA(w).
\end{align}
By Theorem \ref{submean},
\begin{equation}\label{eq:final}
\norm{C_\psi(K_{s_n,2})}^2\geq C\frac{\int\limits_{\mathbb{T}^\infty}\mathcal{N}_{\psi_\chi}(s_n)\,dm_\infty(\chi)}{\Re s_n}.
\end{equation}
The proof now follows from the equation \eqref{comp0},
\begin{equation*}
\lim\limits_{\Re w\rightarrow0}\frac{\int\limits_{\mathbb{T}^\infty}\mathcal{N}_{\psi_\chi}(w)\,dm_\infty(\chi)}{\Re w}=0.\qedhere
\end{equation*}
\end{proof}

\section{Proof of Theorem \ref{characterization}}
The proof of Theorem \ref{characterization} follows from Theorem~\ref{suf1} and Theorem \ref{suf2}. More specifically, let $C_\psi$ be a compact composition operator with symbol $\psi\in\mathfrak{G}_{\geq 1}$. Then, by Theorem~\ref{suf2} $C_\psi$ is compact on $\mathcal{H}_{\Lambda}^2$. Theorem~\ref{suf1} remains true if we substitute $\mathcal{H}^2$ with  $\mathcal{H}_{\Lambda}^2$. The symbol $\psi\in\mathfrak{G}_{\geq1}$ does not depend on the generalized prime $q_1$ and thus
	\begin{equation*}
\lim\limits_{\Re w\rightarrow0}\int\limits_{\mathbb{T}^\infty}\frac{\mathcal{N}_{\psi_\chi}(w)}{\Re w}\,dm_\infty(\chi)=0.
\end{equation*}
Note that in order to prove Theorem~\ref{characterization} it would be sufficient to add just one generalized prime, for example $q=\pi$.

Now we present the counterexample of F. Bayart. We will make use of the following characterization of compact composition operators with symbols $\psi(s)=c_0s+\psi(s)\in\mathfrak{G}_{\geq 1}$, where $\phi$ is a Dirichlet polynomial. 
\begin{theorem}[\cite{BB16}]
Let $\psi(s)=c_0s+\phi(s)\in \mathfrak{G}_{\geq 1}$, where $\phi$ is a Dirichlet polynomial. Then, the induced composition operator $C_\psi$ is compact on $\mathcal{H}^2$ if and only if the symbol has restricted range.
\end{theorem}
We say that a symbol $\psi\in\mathfrak{G}$ has unrestricted range if
\begin{equation}
\inf\limits_{s\in\mathbb{C}_0}\Re\phi(s)=\left\{
\begin{array}{ll}
\frac{1}{2} & \text{ if } \,c_0=0, \\
0 & \text{ if } \,c_0\geq1 .
\end{array} \right.
\end{equation}

It is worth mentioning that a symbol with restricted range always induces a compact composition operator on $\mathcal{H}^2$, \cite[Theorem 20, Theorem 21]{BAY02}.
 
\begin{example}\label{exmp}
We consider the symbol $\psi(s)=1+s-2^{-s}$. The composition operator $C_\psi$ is not compact on $\mathcal{H}^2$, since $\psi(s)=1+s-2^{-s}$ has unrestricted range. The vertical translations of it have the following form
\begin{equation*}
\psi_z(s)=1+s-z2^{-s},\qquad z\in\mathbb{T}.
\end{equation*}  
The function $h(z):=\inf\limits_{\Re s>0}|\psi_z(s)|$ is continuous on $z$ and vanishes only for $z=1$. For $\varepsilon>0$ sufficiently small there exists a constant $C(\varepsilon)>0$ such that
\begin{equation*}
h(z)\geq 2\varepsilon,\qquad |z-1|>C(\varepsilon)
\end{equation*}
and $C(\varepsilon)\rightarrow 0^+$, as $\varepsilon\rightarrow0^+$. Applying Theorem~\ref{bayartL1}, we have that
\begin{equation*}
\int\limits_{\mathbb{T}}\frac{N_{\psi_z}(\varepsilon)}{\varepsilon}\,dz=\int\limits_{1-C(\varepsilon)}^{1+C(\varepsilon)}\frac{N_{\psi_z}(\varepsilon)}{\varepsilon}\,dz\leq 2 C(\varepsilon).
\end{equation*}
Thus
\begin{equation*}
\lim\limits_{\Re w\rightarrow 0}\int\limits_{\mathbb{T}}\frac{N_{\psi_z}(w)}{\Re w}\,dz=0.
\end{equation*}
\end{example}
\bibliographystyle{amsplain-nodash} 
\bibliography{refff} 
\end{document}